\documentclass[11pt,twoside]{article}
\usepackage{amsmath,amssymb,epsfig}
\usepackage{amsmath,amssymb,epsfig,color,graphicx,subfigure}
\usepackage{multirow}
\usepackage{algorithm,algorithmic}
\newtheorem{proposition}{Proposition}[section]
\newtheorem{theorem}[proposition]{Theorem}
\newtheorem{lemma}[proposition]{Lemma}
\newtheorem{corollary}[proposition]{Corollary}

\newtheorem{remark}[proposition]{Remark}
\newtheorem{example}[proposition]{Example}

\newenvironment{proof}{{\noindent \bf Proof:}}{\hfill $\fbox{}$ \vspace*{5mm}}
\numberwithin{equation}{section}

\newcommand{\ba}{{\bf a}}
\newcommand{\bb}{{\bf b}}

\newcommand{\bfe}{{\bf e}}
\newcommand{\bfo}{{\bf 1}}
\newcommand{\bff}{{\bf f}}
\newcommand{\bg}{{\bf g}}

\newcommand{\br}{{\bf r}}

\newcommand{\bx}{{\bf x}}
\newcommand{\by}{{\bf y}}
\newcommand{\bz}{{\bf z}}

\newcommand{\bv}{{\bf v}}
\newcommand{\bw}{{\bf w}}

\newcommand{\cf}{{\cal F}}

\newcommand{\cm}{{\cal M}}

\newcommand{\cs}{{\cal S}}

\newcommand{\cz}{{\cal Z}}

\newcommand{\la}{\lambda}

\newcommand{\R}{{\mathbb{R}}}
\newcommand{\Rm}{{\mathbb{R}^m}}
\newcommand{\Rn}{{\mathbb{R}^n}}
\newcommand{\RN}{{\mathbb{R}^N}}

\newcommand{\RmN}{{\mathbb{R}^{m\times N}}}

\newcommand{\supp}{{\rm supp}}

\newcommand{\ve}{{\rm vec}}

\newcommand{\BE}{\begin{equation}}
\newcommand{\EE}{\end{equation}}
\DeclareMathOperator*{\argmax}{argmax}
\DeclareMathOperator*{\argmin}{argmin}
\newcommand{\normmm}[1]{{\vert\kern-0.25ex \vert\kern-0.25ex \vert #1
    \vert\kern-0.25ex \vert\kern-0.25ex\vert}}

\begin{document}
\title{\bf A Modified Orthogonal Matching Pursuit for  Construction of Sparse Probabilistic Boolean Networks}
\author{Guiyun Xiao\thanks{School of Mathematical Sciences, Xiamen University, Xiamen 361005, People's Republic of China  (xiaogy999@163.com).}
\and Zheng-Jian Bai\thanks{Corresponding author. School of Mathematical Sciences and Fujian Provincial Key Laboratory on Mathematical Modeling \& High Performance Scientific Computing,  Xiamen University, Xiamen 361005, People's Republic of China (zjbai@xmu.edu.cn). The research of this author was partially supported by the National Natural Science Foundation of China (No. 11671337).}
\and Wai-Ki Ching\thanks{Advanced Modeling and Applied Computing Laboratory, Department of Mathematics, The University of Hong Kong, Pokfulam Road, Hong Kong (wching@hku.hk). Research supported in part by Hong Kong RGC GRF Grant no. 17301519, IMR and RAE Research fund from Faculty of Science, HKU.} }
\maketitle

\begin{abstract}
Probabilistic Boolean Networks play a remarkable role in the modelling and control of gene regulatory networks. In this paper, we consider the inverse problem of constructing a sparse probabilistic Boolean network from the prescribed transition probability matrix. We propose a modified orthogonal matching pursuit for solving the inverse  problem. We provide some conditions under which the proposed algorithm can recover a sparse probabilistic Boolean network.  We also report some numerical results to  illustrate the effectiveness of the proposed algorithm.
\end{abstract}

\vspace{3mm}

\section{Introduction}\label{sec1}

\subsection{Boolean Networks and probabilistic Boolean networks}
Boolean Network (BN) and Probabilistic Boolean Network (PBN) arise in a wide variety of applications. The BN model was originally proposed by Kauffmann in 1969 for exploring dynamical properties of  gene regulatory networks \cite{K69} (see also \cite{K69b,K93}). The BN model has been used in different biological systems, including  apoptosis, the yeast cell-cycle network, and T Cell Signaling, and so on (see for instance \cite{B08,LL04,SS07,SS09}).

As an extension of the BN, the PBN has gained much attention since it introduces uncertainty principles into a rule-based BN modelling \cite{SD10,SD02,SD02b}. The PBN model was originally proposed by Shmulevich et al. in 2002 for modelling genetic regulatory networks \cite{SD02b}. The PBN has been used in many applications such as biological systems (see for instance \cite{MW08,SD10}), biomedicine \cite{TM13}, credit defaults \cite{GC13}, and industrial machine systems \cite{RS18, RS18b}, etc.

In the following, we give the basic framework of  BNs and PBNs. As noted in \cite{SD02,SD02b}, a BN  includes a set of nodes (genes)
$V=\{ v_1, v_2, \ldots,v_n \}$ and a list of Boolean functions $F=\{ f_1, f_2, \ldots, f_n\}$. Here, for any $1\le i\le n$, $v_i(t)\in\{ 0,1\}$ is a binary variable, which means the  state (off/on) of gene $i$ at time $t$, and  $f_i:\{ 0,1\}^{n}\to \{ 0,1\}$ is  a Boolean function. The state  of gene $v_i$ at time $t+1$ is determined by
\[
v_i(t+1)= f_i (v_{i_1}(t), \ldots, v_{i_{w(i)}}(t))\equiv f_i (\bv(t)),
\]
where $w(i)$ is the number of essential variables of $f_i$ and $\bv(t)=(v_1(t),\ldots,v_n(t))^T\in\Rn$.
Therefore,  there are  $2^n$ possible global states in a BN with $n$ genes.

In a PBN with a set of  nodes (genes) $V=\{ v_1, v_2, \ldots,v_n \}$, for each gene $v_i$, there exist $l(i)$ possible functions: $F_i=\{f_1^{(i)},\ldots,f_{l(i)}^{(i)}\}$, where each  $f_p^{(i)}$ is a possible function determining the value of gene $v_i$. A realization of the PBN consists of  $N$ different possible realizations, which is determined by $N$ vector  functions $\bff_1,\ldots,\bff_N$ of the form
\BE\label{fi}
\bff_j=(f^{(1)}_{j_1},f^{(2)}_{j_2},\ldots,f^{(n)}_{j_n})^T, \quad j=1,\ldots,N,\quad 1\leq p_i\leq l(i)
\EE
where $f^{(i)}_{j_{i}}\in F_i$ for $i=1,\ldots, n$.

Suppose $\bff=(f^{(1)},\ldots, f^{(n)})^T$ is a random vector with $f^{(i)}\in F_i$. Then, the selection probability of function  $f^{(i)}=f_j^{(i)}$ for gene $v_i$ is given by
\[
c^{(i)}_{j} = {\rm Prob}\{f^{(i)}=f_j^{(i)}\}
\]
for $j=1,\ldots, l(i)$ and  $\sum_{j=1}^{l(i)} c^{(i)}_{j} =1$. Assume that the random variables $f^{(1)},\ldots, f^{(n)}$ are independent. Then, the PBN is called independent. In this case, the probability of choosing the  vector function $\bff_j$ in the form of  (\ref{fi})  is given by
\[
x_j= {\rm Prob}\{\bff=\bff_j\} =\prod_{i=1}^n {\rm Prob}\{f^{(i)}=f_{j_i}^{(i)}\}\equiv\prod_{i=1}^nc^{(i)}_{j_i}.
\]
Therefore, an independent PBN  includes a set of nodes $V$ and a list $\cf=\{F_1,\ldots,F_n\}$, which has $N=\prod_{i=1}^n l(i)$ possible realizations. We note that the independent PBN still has $2^n$ possible global states and the transition probability from state $\ba=(a_1,\ldots,a_n)^T$ to state $\bb=(b_1,\ldots,b_n)^T$  is determined by
\begin{eqnarray*} \label{P}
&& {\rm Prob} \{\bv(t+1)= \bb \ | \ \bv(t)= \ba \} \nonumber \\
&=& \sum_{j=1}^N \ {\rm Prob} \ \{ \bv(t+1)= \bb \ | \ \bv(t)=\ba, \mbox{the $j$th vector function (\ref{fi}) is selected}\} \cdot x_j.
\end{eqnarray*}
Then we obtain  the transition probability matrix $P\in\R^{2^n\times 2^n}$ of the PBN  \cite{CZ07}:
\[
P = \sum_{j=1}^N x_j A_j,
\]
where  $A_j\in\R^{2^n\times 2^n}$ is  the transition probability matrix  corresponding to the $j$th constituent vector function $\bff_j$. Here, $\R^{n_1\times n_2}$ is the set of all $n_1\times n_2$ real matrices ($\R^{n}=\R^{n\times 1}$) and each column of $A_j$ has only one nonzero entry and each column adds up to one.

\subsection{Construction of probabilistic Boolean network}

The inverse problem of constructing a PBN aims to identify all the  constituent BNs and corresponding selection probabilities such that the constructed PBN has the prescribed  transition probability matrix. Suppose a PBN consists of $N$ possible constituent BNs with the transition probability matrices  $\{A_j\}_{j=1}^N$. The inverse problem of constructing a PBN aims to  find the probability distribution vector $\bx=(x_1,\ldots, x_N)^T$ from the prescribed transition probability matrix $P$ and the constituent  BN matrices $\{A_j\}_{j=1}^N$ such that
\BE\label{pbn}
P=\sum_{j=1}^{N}x_jA_j, \quad \bfo^T \bx=1, \quad \bx\ge {\bf 0},
\EE
where $\bfo$ is a column vector of an appropriate dimension whose entries are all ones and for any two vectors $\bff,\bg\in\RN$, $\bg\ge\bff$ means that $g_j\ge f_j$ for $j=1,\ldots,N$. 

One may solve \eqref{pbn}  by the solution of the following minimization problem:
\BE\label{pbn:ls}
\begin{array}{lc}
 \min\limits_{\bx\in\RN} &  \displaystyle \frac{1}{2}\|P-\sum_{j=1}^{N}x_jA_j\|_F^2  \\
\mbox{subject to (s.t.)} & \bfo^T \bx=1,\quad \bx\ge  {\bf 0},
 \end{array}
\EE
where $\|\cdot\|_F$ denotes the  Frobenius norm.
Let
\BE\label{def:wu}
A=[\ve(A_1),\ve(A_2),\ldots,\ve(A_N)]\in\RmN
\quad\mbox{and}\quad
\bb=\ve(P)\in\Rm,
\EE
where $m=2^{2n}\ll N$ and $\ve(\cdot)$  generates a column vector from a matrix by stacking its column vectors below one another. Then the minimization problem \eqref{pbn:ls} takes the form of
\BE\label{pbn:ls-s}
\begin{array}{lc}
\min\limits_{\bx\in\RN}  &  \displaystyle \frac{1}{2}\|A\bx-\bb\|_2^2  \\
\mbox{s.t.} &\bfo^T \bx=1,\quad \bx\ge  {\bf 0}.
 \end{array}
\EE
where $\|\cdot\|_2$ denotes the Euclidean vector norm or its induced matrix norm.

In general, there are many solutions to the inverse problem. However, in practice, it is desired to find only a few major  constituent BNs with associated selection probabilities. That is, a sparse solution to the inverse problem gives a simple approximate PBN, which may provide a good control design for gene regulatory networks.  To find a sparse solution to problem \eqref{pbn:ls-s}, one may solve the  following $\ell_{0}$ regularization problem:
\BE\label{pbn:l0}
\begin{array}{lc}
\min\limits_{\bx\in\RN}  &  \displaystyle \frac{1}{2}\|A\bx-\bb\|_2^2 +\lambda \|\bx\|_0 \\
\mbox{s.t.} &\bfo^T \bx=1,\quad \bx\ge {\bf 0},
 \end{array}
\EE
where $\lambda>0$ is a regularization parameter and $\|\cdot\|_0$ means the number of nonzero entries of a vector. However, this is a NP-hard problem \cite{N95}. It is natural to consider the following  $\ell_1$-norm relaxed version of problem \eqref{pbn:l0}:
\BE\label{pbn:l1}
\begin{array}{lc}
\min\limits_{\bx\in\RN}  &  \displaystyle \frac{1}{2}\|A\bx-\bb\|_2^2 +\lambda \|\bx\|_1 \\
\mbox{s.t.} & \bfo^T \bx=1,\quad \bx\ge {\bf 0},
 \end{array}
\EE
There is a large literature on the solution of such convex minimization problem. However, it seems invalid to adopt  the  $\ell_1$  regularization for problem \eqref{pbn:ls-s} since  the equality constraint $\bfo^T \bx=1$ is equivalent to the $\ell_1$-norm  regularization term $\|\bx\|_1=1$ due to $\bx\ge 0$.

There exists many methods for finding a sparse solution to the inverse problem. For instance, a heuristic algorithm was proposed in \cite{CC08}. A dominant modified algorithm was proposed in \cite{CL10}. A maximum entropy rate approach and its modified version were proposed in \cite{CC11,CJ12,CC09}.  A projection-based gradient descent method was presented in \cite{WW15}.

Recently, an alternating direction method of multipliers  was given in \cite{LP14} for solving the following non-convex minimization problem with the $\ell_{1/2}$ regularization:
\[
\begin{array}{lc}
\min\limits_{\bx\in\RN}  &  \displaystyle \frac{1}{2}\mu\|A\bx-\bb\|_2^2 +\sum_{j=1}^{N}x_j\log x_j +\lambda \|\bx\|_{1/2}^{1/2} \\
\mbox{s.t.} &\bfo^T \bx=1,\quad \bx\ge {\bf 0},
 \end{array}
\]
where $\mu$ and $\la$ are two positive constants.
In \cite{DP19}, a partial proximal-type operator splitting method was proposed  for solving the $\ell_{1/2}$ regularization version of problem \eqref{pbn:l0}:
\[
\begin{array}{lc}
\min\limits_{\bx\in\RN}  &  \displaystyle \frac{1}{2}\|A\bx-\bb\|_2^2 +\lambda \|\bx\|_{1/2}^{1/2} \\
\mbox{s.t.} &\bfo^T \bx=1,\quad \bx\ge {\bf 0},
 \end{array}
\]
where $\la>0$ is a constant.
\subsection{Our contribution}
The orthogonal matching pursuit (OMP) is a greedy algorithm for solving the sparse approximation problem over a redundant dictionary, which was introduced independently in many references (see for instance \cite{CB89,DM94,PR93}). The sparse recovery of the OMP was analyzed by Tropp in \cite{T04} and was extended to the noise case \cite{CW11}. The OMP aims to find a sparse solution to an underdetermined linear system of linear equations $\by=\Phi\bw$, where $\Phi$ is a $q\times Q$ matrix with $q<Q$. However, the OMP can not be directly applied to finding a sparse solution to problem \eqref{pbn:ls-s} since there exist additional nonnegative constraint $\bx\ge{\bf 0}$ and equality constraint $\bfo^T\bx=1$.

In this paper, we propose a modified orthogonal matching pursuit (MOMP)  for finding a sparse solution to problem \eqref{pbn:ls-s}. By exploring  the properties of the $m\times N$ matrix $A$ and  the vector $\bb\in\Rm$ defined by \eqref{def:wu}, we give some conditions to guarantee that our method can find a sparse solution to problem \eqref{pbn:ls-s}.  We also present some  numerical examples to illustrate the efficiency of our method for constructing a sparse PBN.

\subsection{Organization} The rest of this paper is organized as follows. In Section
\ref{sec2}, we review the OMP and then propose a MOMP for constructing a sparse PBN. In Section \ref{sec3}, we discuss the convergence analysis of our method. In Section \ref{sec4}, we present some numerical examples to show the efficiency of the proposed method. Finally, we give some concluding remarks in Section \ref{sec5}.
\subsection{Notation}
Throughout this paper, we use the following notation. Let $I$ be the identity matrix of  an appropriate dimension. Denote by $\bfe_j$ the $j$-th column of $I$. The superscripts ``$\cdot^T$" denotes the transpose  of a matrix.   For any $A\in\RmN$, let $A=[\ba_1,\ldots,\ba_N]$.
For a complex number $a$, $|a|$ denotes the modulus of $a$.
Let $[N]=\{1,2,\ldots,N\}$ and for any set $\cs\subset[N]$, let $|\cs|$ and $[N]\backslash\cs$ be the cardinality of $\cs$ and the complement of $\cs$ in $[N]$, respectively.
For any set $\cs\subset[N]$, $A_\cs$ is the submatrix of a matrix $A$ with columns indexed by $\cs$.
A vector $\bz$ is called $d$-sparse if at most $d$ entries of $\bz$ are nonzero. Finally, denote by  $\supp(\bz):=\{j\in[N]\; |\; z_j\neq 0\}$ the support of a vector $\bz\in\RN$.

\section{A modified orthogonal matching pursuit} \label{sec2}

In this section, we first recall the  OMP for solving underdetermined linear systems. Then we propose a MOMP for solving problem \eqref{pbn:ls-s}.

\subsection{Orthogonal matching pursuit}
The OMP aims to find a sparse solution to the following underdetermined linear system:
\BE\label{le:omp}
\by=\Phi\bw,
\EE
where $\Phi\in\R^{q\times Q}$ is a measurement matrix with $q<Q$ and $\by\in\R^q$ is the  observation vector.  Then the OMP algorithm is stated as in Algorithm \ref{ap0}.
\begin{algorithm}[h]
\caption{OMP for problem (\ref{le:omp})} \label{ap0}
\begin{description}
\item [{\rm Step 0.}] Choose an initial point $\bw^0={\bf 0}$ and $\cs^0=\emptyset$. Let $k:=0$.
\item [{\rm Step 1.}] Find $j_{k+1}\in [N]$ such that
\[
j_{k+1}\in\argmax_{j\in[N]}|\bfe_j^T\Phi^T(\by-\Phi\bw^k)|.
\]
Set $\cs^{k+1}= \cs^k\cup\{j_{k+1}\}$.
\item [{\rm Step 2.}]  Find
\[
\bw^{k+1}=\argmin_{\bw\in\R^Q \;\supp(\bx)\subset \cs^{k+1}} \frac{1}{2}\|\by-\Phi\bw\|_2^2.
\]
\item [{\rm Step 3.}] Replace $k$ by $k+1$ and go to  Step 1.
\end{description}
\end{algorithm}

We see that the OMP algorithm is simple and easy to implement. For more details on the OMP, one may refer to  \cite{CB89,DM94,PR93,T04}. In particular, one may refer to \cite[Proposition 3.5]{FR13} for the exact recovery condition for the OMP.

\subsection{A modified orthogonal matching pursuit}
In this subsection, we propose a MOMP for solving problem \eqref{pbn:ls-s}. It is natural to extend the OMP (i.e., Algorithm \ref{ap0}) to the solution of
problem \eqref{pbn:ls-s}. Compared with problem (\ref{le:omp}), we have additional equality constraint  $\bfo^T\bx=1$ and  nonnegative constraint $\bx\ge{\bf 0}$. Hence, we cannot solve problem \eqref{pbn:ls-s} by the OMP directly. We also note that, for any $1\le j\le N$,    each column of the $j$-th constituent BN matrix $A_j\in\R^{2^n\times 2^n}$ has only one nonzero entry and each column adds up to one. Thus the matrix  $A$ defined by \eqref{def:wu} is entrywise nonnegative, sparse, and satisfies  the property
\BE\label{a:p}
\bfo\ge A\bx,\quad \forall \bx\in\cm,
\EE
where $\cm$ is the feasible domain of  problem \eqref{pbn:ls-s}, which is defined by
\[
\cm:=\Big\{\bx\in\RN\; |\; \bfo^T \bx=1,\; \bx\ge {\bf 0}\Big\}.
\]
In addition, we see that  the prescribed  transition probability matrix $P\in\R^{2^n\times 2^n}$ is usually sparse. Hence, the vector $\bb\in\Rm$ defined by \eqref{def:wu} satisfies  the following property
\BE\label{b:p}
\bfo\ge\bb\ge {\bf 0}.
\EE

From the above analysis, sparked by the OMP  (i.e., Algorithm \ref{ap0}), we propose a MOMP for solving problem (\ref{pbn:ls-s}). The algorithm  is described in Algorithm \ref{ap}.
\begin{algorithm}[h]
\caption{MOMP for problem (\ref{pbn:ls-s})} \label{ap}
\begin{description}
\item [{\rm Step 0.}] Choose an initial guess $\bx^0\in\cm$  and $\cs^0=\emptyset$. Let $k:=0$.
\item [{\rm Step 1.}] Find $j_{k+1}\in [N]$ such that
\[
j_{k+1}\in\argmax_{j\in[N]}\bfe_j^TA^T(\bb-A\bx^k).
\]
Set $\cs^{k+1}= \cs^k\cup\{j_{k+1}\}$.
\item [{\rm Step 2.}]  Find
\BE\label{eq:ck1}
\bx^{k+1}=\argmin_{\substack{\bx\in\cm,\;\supp(\bx)\subset \cs^{k+1}}} \frac{1}{2}\|\bb-A\bx\|_2^2.
\EE
\item [{\rm Step 3.}] Replace $k$ by $k+1$ and go to  Step 1.
\end{description}
\end{algorithm}

We point out that  the major work of Algorithm {\rm \ref{ap}} is to solve a small linear least square problem \eqref{eq:ck1}, which can be solved  via  the standard solvers for constrained linear least square problems, e.g., the interior point algorithm or the active-set algorithm (see for instance \cite{NW06}).
\section{Convergence analysis}  \label{sec3}
In this section, we show that Algorithm {\rm \ref{ap}}  converges in finite steps under some conditions.

For the iterate $j_{k+1}$ generated by Algorithm \ref{ap}, we have the following lemma.
\begin{lemma}\label{lem:jk1}
Let  $\bx^k$ be the current iterate generated by Algorithm {\rm \ref{ap}} with $k\ge 1$. If
\[
A\bfe_{j_{k+1}}=A\bx^k,
\]
then $j_{k+1}\in \cs^{k+1}$ is such that $\bx^{k+1}\in\cm $ with $\supp(\bx^{k+1})\subset \cs^{k+1}$ but
\[
\|\bb-A\bx^{k+1}\|_2=\|\bb-A\bx^{k}\|_2.
\]
Moreover, if $\bx^{k+1}=\bx^{k}$, then $j_{k+1}\in \cs^k$.
\end{lemma}
\begin{proof}
We note that $\cs^{k+1}=\cs^k\cup\{j_{k+1}\}$, where $\cs^k\neq \emptyset$ since $k\ge 1$. By hypothesis, $A\bfe_{j_{k+1}}=A\bx^k$. Then, without loss of generality, we have
\BE\label{eq:ask1}
A_{\cs^{k+1}}=[A_{\cs^k}, A\bfe_{j_{k+1}}]=[A_{\cs^k}, A\bx^k]=[A_{\cs^k}, A_{\cs^k}\bx_{\cs^k}^k].
\EE
This means that the last column of $A_{\cs^{k+1}}$ is a convex combination of the  columns of $A_{\cs^k}$.
From \eqref{eq:ask1} we have for all $\bx\in\cm $ with $\supp(\bx)\subset \cs^{k+1}$,
\begin{eqnarray}\label{eq:res}
\|\bb-A\bx\|_2&=&\|\bb-A_{\cs^{k+1}}\bx_{\cs^{k+1}}\|_2
=\|\bb-[A_{\cs^k}, A_{\cs^k}\bx_{\cs^k}^k]\bx_{\cs^{k+1}}\|_2 \nonumber\\
&=& \|\bb-A_{\cs^k}(\bx_{\cs^{k}}+x_{j_{k+1}} \bx_{\cs^k}^k)\|_2.
\end{eqnarray}
For any $\bx\in\cm $ with $\supp(\bx)\subset \cs^{k+1}$, it is easy to see that $\bx_{\cs^{k}}+x_{j_{k+1}} \bx_{\cs^k}^k\ge 0$, $\sum_{i\in \cs^k} \big((\bx_{\cs^{k}})_i+x_{j_{k+1}} (\bx_{\cs^k}^k)_i\big)=1$, and $\supp(\bx_{\cs^{k}}+x_{j_{k+1}} \bx_{\cs^k}^k)\subset \cs^{k}$. Notice
\[
\bx^{k}=\argmin_{\substack{\bx\in\cm,\;\supp(\bx)\subset \cs^{k}}} \frac{1}{2}\|\bb-A\bx\|_2^2.
\]
It follows from \eqref{eq:res}  that
\begin{eqnarray}\label{eq:bxk1}
\|\bb-A\bx^{k+1}\|_2&=&\min_{\bx\in\cm,\; \supp(\bx)\subset \cs^{k+1}}\|\bb-A\bx\|_2^2\nonumber\\
&=&\min_{\bx\in\cm,\; \supp(\bx)\subset \cs^{k+1}} \|\bb-A_{\cs^k}(\bx_{\cs^{k}}+x_{j_{k+1}} \bx_{\cs^k}^k)\|_2\nonumber\\
&=& \|\bb-A\bx^{k}\|_2,
\end{eqnarray}
where the last equality holds by setting $\bx_{\cs^{k}}=(1-x_{j_{k+1}})\bx_{\cs^k}^k$ for all $0\le x_{j_{k+1}}\le 1$.

Moreover, it is easy to see that $\bx^{k+1}=\bx^{k}$ is  a special solution  to \eqref{eq:bxk1}. In this case, we have $j_{k+1}\in \cs^k$.
\end{proof}

The following result shows that the choice of the index $j_{k+1}$ is reasonable in the sense that  the residual is nonincreasing.
\begin{theorem}
Let $\{\bx^k\}$ be the sequence generated by Algorithm {\rm \ref{ap}}. Then we have, for all $k\ge 1$,
\[
\|\bb-A\bx^{k+1}\|_2^2
\left\{
\begin{array}{ll}
= \|\bb-A\bx^k\|_2^2, & \mbox{if $A\bfe_{j_{k+1}}=A\bx^k$},\\[3mm]
\leq \|\bb-A\bx^k\|_2^2 -\frac{\Big(\bfe_{j_{k+1}}^TA^T(\bb-A\bx^k)-(\bx^k)^TA^T(\bb-A\bx^k)\Big)^2}
{\|A(\bfe_{j_{k+1}}-\bx^k)\|_2^2}, & \mbox{otherwise}.
\end{array}
\right.
\]
\end{theorem}
\begin{proof}
For any $0\le t\le 1$, let
\[
\widetilde{\bx}^k:=(1-t)\bx^k+t\bfe_{j_{k+1}}.
\]
It is easy to verify that $\widetilde{\bx}^k\in\cm$ and $\supp(\widetilde{\bx}^k)\subset \cs^{k+1}$.
Thus, for any $0\le t\le 1$,
\begin{eqnarray}\label{res:rk1}
\|\bb-A\bx^{k+1}\|_2^2 &=&\min_{\bx\in\cm,\;\supp(\bx)\subset \cs^{k+1}}\|\bb-A\bx\|_2^2\nonumber\\
&\leq &\| \bb-A\widetilde{\bx}^k \|_2^2 = \| \bb-A\big((1-t)\bx^k+t\bfe_{j_{k+1}}\big)\|_2^2\nonumber\\
&=&  \| (\bb-A\bx^k)-tA(\bfe_{j_{k+1}}-\bx^k)\|_2^2\nonumber\\
&=&\| \bb-A\bx^k\|_2^2+t^2\|A(\bfe_{j_{k+1}}-\bx^k)\|_2^2-2t\langle A(\bfe_{j_{k+1}}-\bx^k),\bb-A\bx^k\rangle.
\end{eqnarray}
If $A\bfe_{j_{k+1}}=A\bx^k$, then using Lemma \ref{lem:jk1} we have
\[
\|\bb-A\bx^{k+1}\|_2=\|\bb-A\bx^{k}\|_2.
\]

We now assume that  $A(\bfe_{j_{k+1}}-\bx^k)\neq {\bf 0}$. From \eqref{res:rk1} we have
\[
\|\bb-A\bx^{k+1}\|_2^2 \le \| \bb-A\bx^k\|_2^2+\|A(\bfe_{j_{k+1}}-\bx^k)\|_2^2\big(t^2-2t\sigma_k\big),
\]
for all $0\le t\le 1$, where
\[
\sigma_k:=\frac{\langle A(\bfe_{j_{k+1}}-\bx^k),\bb-A\bx^k\rangle}{\|A(\bfe_{j_{k+1}}-\bx^k)\|_2^2}
=\frac{\bfe_{j_{k+1}}^TA^T(\bb-A\bx^k)-(\bx^k)^TA^T(\bb-A\bx^k)}{\|A(\bfe_{j_{k+1}}-\bx^k)\|_2^2}.
\]
Thus,
\[
\|\bb-A\bx^{k+1}\|_2^2 \le \| \bb-A\bx^k\|_2^2+\|A(\bfe_{j_{k+1}}-\bx^k)\|_2^2\;\min_{0\le t\le 1}(t^2-2t\sigma_k).
\]

We now show that $0\le \sigma_k\le 1$. We first derive that $\sigma_k\ge 0$. Using the definition of $j_{k+1}$ and $\bx^k\in\cm$ and $\supp(\bx ^k)\subset \cs^k$ we have
\[
(\bx^k)^TA^T(\bb-A\bx^k)\le \bfe_{j_{k+1}}^TA^T(\bb-A\bx^k)\sum_{j\in \cs^k} x_j^k=\bfe_{j_{k+1}}^TA^T(\bb-A\bx^k).
\]
This shows that  $\sigma_k\ge 0$. On the other hand, we note that, if  $\bfe_i^TA\bfe_{j_{k+1}}\neq 0$ for some $1\le i\le m$, then $\bfe_i^TA\bfe_{j_{k+1}}=1$. Thus,
\[
\| A(\bfe_{j_{k+1}}-\bx^k)\|_2^2-\langle A(\bfe_{j_{k+1}}-\bx^k),\bb-A\bx^k\rangle
=\langle A(\bfe_{j_{k+1}}-\bx^k),A\bfe_{j_{k+1}}-\bb\rangle\geq 0,
\]
where the last inequality uses the fact that $\bfe_i^TA(\bfe_{j_{k+1}}-\bx^k) \ge 0$ and $\bfe_i^T(A\bfe_{j_{k+1}}-\bb) \ge 0$ for all $i\in\supp(A\bfe_{j_{k+1}})$ and $\bfe_i^TA(\bfe_{j_{k+1}}-\bx^k) \le 0$ and $\bfe_i^T(A\bfe_{j_{k+1}}-\bb) \le 0$ for all $i\notin\supp(A\bfe_{j_{k+1}})$ by using the properties \eqref{a:p} and  \eqref{b:p}. Therefore, we have $0\le \sigma_k\le 1$. Substituting $t=\sigma_k$ yields
\begin{eqnarray*}
&& \|\bb-A\bx^{k+1}\|_2^2 \le
\| \bb-A\bx^k\|_2^2-\sigma_k^2\|A(\bfe_{j_{k+1}}-\bx^k)\|_2^2\\
&=&  \| \bb-A\bx^k\|_2^2 -\frac{\Big(\bfe_{j_{k+1}}^TA^T(\bb-A\bx^k)-(\bx^k)^TA^T(\bb-A\bx^k)\Big)^2}
{\|A(\bfe_{j_{k+1}}-\bx^k)\|_2^2}.
\end{eqnarray*}
The proof is complete.
\end{proof}

On the optimality conditions of problem \eqref{eq:ck1}, we have the following result from \cite[Theorem 16.4]{NW06}.
\begin{lemma}\label{lem:ck}
Let $\bx^{k+1}$ be the current iterate of  Algorithm {\rm \ref{ap}}. Then $\bx^{k+1}\in\cm$ with $\supp(\bx^{k+1})\subset \cs^{k+1}$  is a global solution to problem \eqref{eq:ck1} if and only if
\[
\Big(A_{\cs^{k+1}}^T\big(\bb-A\bx^{k+1}\big)\Big)_l
\left\{
\begin{array}{ll}
=\big(\bx^{k+1}\big)^TA^T\big(\bb-A\bx^{k+1}\big), & \mbox{if $l\in \supp(\bx^{k+1})$},\\[2mm]
\le\big(\bx^{k+1}\big)^TA^T\big(\bb-A\bx^{k+1}\big), & \mbox{if $l\in \cs^{k+1}\setminus \supp(\bx^{k+1})$}.
\end{array}
\right.
\]
Moreover,  if $A_{\cs^{k+1}}: \cz_{\cs^{k+1}}\to\Rm$ is injective, then $\bx^{k+1}\in\cm$ with $\supp(\bx^{k+1})\subset \cs^{k+1}$  is the unique global solution to problem \eqref{eq:ck1}, where $\cz_{\cs^{k+1}}:=\{\bz\in\R^{|\cs^{k+1}|} \; |\; \bfo^T \bz=1,\; \bz\ge 0\}$.
\end{lemma}

We now discuss the convergence conditions for the MOMP. We first  give some necessary conditions for Algorithm \ref{ap} to  recover a sparse solution to the linear system $\bb=A\bx$. The proof can be seen as a generalization of \cite[Proposition 3.5]{FR13}.
\begin{theorem}\label{thm:nec}
Let $A\in \RmN$ and $\bb\in\Rm$ be defined by \eqref{def:wu}. Suppose  every nonzero vector $\bx^*\in\cm$ supported on a set $\cs$ of size $d$  is recovered from $\bb=A\bx^*$ via  Algorithm {\rm \ref{ap}} with any fixed starting point $\bx^0\in\cm$  after at most $d$ iterations. Then the linear operator $A_\cs:\cz_\cs\to\Rm$ is injective,
\BE\label{nec1}
\max_{j\in \cs}\big(A^T(\bb-A\bx^0)\big)_j > \max_{l\in [N]\backslash\cs}\big(A^T(\bb-A\bx^0)\big)_l,
\EE
for all $\bb\in\{A\bx \; |\; \bx\in\cm,\; \supp(\bx)\subset \cs\}$, where $\cz_\cs:=\{\bz\in\R^{|\cs|} \; |\; \bfo^T \bz=1,\; \bz\ge 0\}$.
\end{theorem}
\begin{proof}
Suppose Algorithm \ref{ap} recovers all vectors supported on a set $\cs$ of size $d$ at most $d$ iterations. Then, for  any two vectors $\bx_1,\bx_2\in\cm$ supported on $\cs$ with $A\bx_1=\bb=A\bx_2$, we must have  $\bx_1=\bx_2$. This shows that  the linear operator $A_\cs:\cz_\cs\to\Rm$  is injective.
On the other hand, if there exists a vector $\bx^*\in\cm$ with $\supp(\bx^*)\subset\cs$ such that $\bb=A\bx^*$, then the index $j_1$ generated by  Algorithm \ref{ap} at the first iteration should not belong to $[N]\backslash\cs$, i.e., $\max_{j\in \cs}\big(A^T(\bb-A\bx^0)\big)_j > \max_{l\in [N]\backslash\cs}\big(A^T(\bb-A\bx^0)\big)_l$. Therefore, we have $\max_{j\in \cs}(A^T(\bb-A\bx^0))_j > \max_{l\in [N]\backslash\cs}(A^T(\bb-A\bx^0))_l$ for all $\bb\in\{A\bx \; |\; \bx\in\cm,\; \supp(\bx)\subset \cs\}$. This completes the proof.
\end{proof}

Next, we provide some sufficient conditions to guarantee Algorithm \ref{ap} recovers all sparse solutions of the linear system $\bb=A\bx$ exactly. The proof can be seen as a generalization of \cite[Proposition 3.5]{FR13}.
\begin{theorem}\label{thm:sc}
Let $A\in \RmN$ and $\bb\in\Rm$ be defined by \eqref{def:wu}. Then  every nonzero vector $\bx^*\in\cm$ supported on a set $\cs$ of size $d$  is recovered from $\bb=A\bx^*$ via  Algorithm {\rm \ref{ap}} with any fixed starting point $\bx^0\in\cm$   after at most $d$ iterations if the linear operator $A_\cs:\cz_\cs\to\Rm$ is injective,
\BE\label{sc1}
\max_{j\in \cs}\big(A^T(\bb-A\bx^0)\big)_j > \max_{l\in [N]\backslash\cs}\big(A^T(\bb-A\bx^0)\big)_l,
\EE
and
\BE\label{sc2}
\max_{j\in \cs}\big(A^T(\bb-A\bx)\big)_j > \max_{l\in [N]\backslash\cs}\big(A^T(\bb-A\bx)\big)_l,
\EE
for all $\bx\in\cm_\cs:=\{\bx\in\cm \; |\; \supp(\bx)\subset \cs\}\backslash \{\bx^*\}$, where $\cz_\cs:=\{\bz\in\R^{|\cs|} \; |\; \bfo^T \bz=1,\; \bz\ge 0\}$.
\end{theorem}
\begin{proof}
Suppose the starting point $\bx^0\in \cm$  with $\supp(\bx^0)\not\subset\cs$ is such that $\bb=A\bx^0$. This contradicts (\ref{sc1}).
We now assume that $\bb\neq A\bx^k$ for $k=1,\ldots, d-1$ (otherwise, we have found the solution).
We claim that, for any $1\le k\le d$, $\cs^k\subset \cs$ is of size $k$. Therefore $\cs=\cs^d$ and $\bx^*=\bx^d$ since  the linear operator $A_\cs:\cz_\cs\to\Rm$  is injective. In the following, we show the claim  by the induction. We first show that, for any $1\le k\le d$, $\cs^k\subset\cs$ (which implies that $\bx^k\in \cm$  with $\supp(\bx^k)\subset\cs^k$).  Using (\ref{sc1}), we know that the first index $j_1$ must belong to $\cs$ and thus $\cs^1=\cs^0\cup \{j_1\}= \{j_1\}\subset \cs$. Now, suppose $\cs^k\subset\cs$ for some $1\le k\le d-1$. Then, using \eqref{sc2} we have the index $j_{k+1}\in \cs$ and thus $\cs^{k+1}=\cs^k\cup\{j_{k+1}\}\subset \cs$. By the induction, we have $S^k\subset S$ for all $1\le k\le d$. Next, we show that $\cs^k$ is of size $k$ for all  $1\le k\le d$. For any $1\le k\le d$, using Lemma \ref{lem:ck} we have
\[
\Big(A_{\cs^{k}}^T\big(\bb-A\bx^{k}\big)\Big)_l
\left\{
\begin{array}{ll}
=\big(\bx^{k}\big)^TA^T\big(\bb-A\bx^{k}\big), & \mbox{if $l\in \supp(\bx^{k})$},\\[2mm]
\le\big(\bx^{k}\big)^TA^T\big(\bb-A\bx^{k}\big), & \mbox{if $l\in \cs^{k}\setminus \supp(\bx^{k})$}.
\end{array}
\right.
\]
By definition, $j_{k+1}=\argmax_{j\in[N]}\bfe_j^TA^T(\bb-A\bx^k)\notin \cs^k$. Otherwise,  if $j_{k+1}\in \cs^k$, then it follows from  \eqref{sc2} that
\[
\max_{j\in \cs^k}\big(A^T(\bb-A\bx^k)\big)_j=\max_{j\in \cs}\big(A^T(\bb-A\bx^k)\big)_j > \max_{l\in [N]\backslash\cs}\big(A^T(\bb-A\bx^k)\big)_l.
\]
Thus,
\[
\Big(A^T\big(\bb-A\bx^{k}\big)\Big)_l
\left\{
\begin{array}{ll}
=\big(\bx^{k}\big)^TA^T\big(\bb-A\bx^{k}\big), & \mbox{if $l\in \supp(\bx^{k})$},\\[2mm]
\le\big(\bx^{k}\big)^TA^T\big(\bb-A\bx^{k}\big), & \mbox{if $l\in \cs^{k}\setminus \supp(\bx^{k})$},\\[2mm]
\le\big(\bx^{k}\big)^TA^T\big(\bb-A\bx^{k}\big), & \mbox{if $l\in \cs\setminus \cs^{k}$},\\[2mm]
<\big(\bx^{k}\big)^TA^T\big(\bb-A\bx^{k}\big), & \forall\; l\in[N]\backslash\cs.
\end{array}
\right.
\]
Using Lemma \ref{lem:ck} and the injectivity of the linear operator $A_\cs:\cz_\cs\to\Rm$, we know that  $\bx^k$ is the unique global solution to problem \eqref{pbn:ls-s}. By assumption, $\bx^*\in\cm$ with $\supp(\bx^*)\subset\cs$ is such that $\bb=A\bx^*$, which is a global solution to problem \eqref{pbn:ls-s}. Thus $\bx^*=\bx^k$. This is a  contradiction. Therefore, $\cs^k$ is of size $k$. The proof is complete.
\end{proof}

\begin{remark}\label{rem:nscond}
We observe that the necessary conditions in Theorem {\rm \ref{thm:nec}} are not equivalent to the sufficient conditions in Theorem {\rm\ref{thm:sc}}.  This may be caused by the additional constraints:  $\bfo^T\bx=1$ and $\bx\ge{\bf 0}$. By assumptions, $\bb=A\bx^*$ for every exact recovery $\bx^*\in\cm$ supported on a set $\cs$ of size $d$. Then $\bb$ belongs to the set $\{A\bx \; |\; \bx\in\cm,\; \supp(\bx)\subset \cs\}$. While, for any $1\le k\le d-1$, we have $\cs^k\subset \cs$ but the residual $\br^k:=\bb-A\bx^k=A(\bx^*-\bx^k+\bx^0)-A\bx^0$, where $A(\bx^*-\bx^k+\bx^0)$ is not guaranteed to belong to the set $\{A\bx \; |\; \bx\in\cm,\; \supp(\bx)\subset \cs\}$ since  the support of $\bx^0$ is not necessary on $\cs$ and the entrywise nonnegativity of the vector $(\bx^*-\bx^k+\bx^0)$ is not guaranteed.
\end{remark}

By following the similar proof of  Theorem \ref{thm:sc}, we have the following sufficient conditions on the sparse recovery of Algorithm {\rm \ref{ap}} for problem \eqref{pbn:ls-s}.
\begin{theorem}\label{thm:sc-ls}
Let $A\in \RmN$ and $\bb\in\Rm$  be defined by \eqref{def:wu}.  Then every nonzero vector $\bx^*\in\cm$ supported on a set $\cs$ of size $d$  solve problem \eqref{pbn:ls-s} via  Algorithm {\rm \ref{ap}} after at most $d$ iterations if the linear operator $A_\cs:\cz_\cs\to\Rm$ is injective,
\[
\max_{j\in \cs}\big(A^T(\bb-A\bx^0)\big)_j > \max_{l\in [N]\backslash\cs}\big(A^T(\bb-A\bx^0)\big)_l,
\]
and
\[
\max_{j\in \cs}\big(A^T(\bb-A\bx)\big)_j > \max_{l\in [N]\backslash\cs}\big(A^T(\bb-A\bx)\big)_l,
\]
for all $\bx\in\cm_\cs$, where $\cm_\cs$ and $\cz_\cs$ are defined as in Theorem {\rm \ref{thm:sc}}.
\end{theorem}

\begin{remark}\label{rem:2}
In Theorems {\rm \ref{thm:nec}}--{\rm \ref{thm:sc-ls}}, we require that the injectivity of the linear operator $A_\cs:\cz_\cs\to\Rm$, which is guaranteed if $A_\cs:\R^{|\cs|}\to\Rm$ is injective i.e., $A_\cs$ is full column rank. We note that $\cz_\cs$ is a closed convex subset of  $\R^{|\cs|}$. It is easy to see that if $|\cs|>m$, then the linear operator $A_\cs:\R^{|\cs|}\to\Rm$ cannot be injective. This shows that, if Algorithm {\rm \ref{ap}} generates a sparse solution to problem \eqref{pbn:ls-s}, then the sparsity is no more than $m$. 
\end{remark}

Based on   Theorems {\rm \ref{thm:sc}}--{\rm \ref{thm:sc-ls}} and Remark \ref{rem:2}, for Algorithm {\rm \ref{ap}}, we have the following results on the recovery  with a given support for problem \eqref{pbn:ls-s}  exactly or in the least square sense.
\begin{corollary}\label{cor:ex}
Let $A\in \RmN$ and $\bb\in\Rm$  be defined by \eqref{def:wu}.  Let $\cs\subset [N]$ with $d=|\cs|$. Then every nonzero vector $\bx^*\in\cm$  with $\supp(\bx^*)\subset\cs$  is recovered from  $\bb=A\bx^*$ via Algorithm {\rm \ref{ap}}  after at most $d$ iterations if  $A_\cs$ has full column rank and
\[
\max_{j\in \cs}\big(A^T(\bb-A\bx^k)\big)_j > \max_{l\in [N]\backslash\cs}\big(A^T(\bb-A\bx^k)\big)_l,
\]
for $k=0,1,\ldots, d-1$.
\end{corollary}
\begin{corollary}\label{cor:inex}
Let $A\in \RmN$ and $\bb\in\Rm$  be defined by \eqref{def:wu}.  Let $\cs\subset [N]$ with $d=|\cs|$. Then every nonzero vector $\bx^*\in\cm$  with $\supp(\bx^*)\subset\cs$ solve problem \eqref{pbn:ls-s}  in the least square sense via  Algorithm {\rm \ref{ap}} after at most $d$ iterations if $A_\cs$ has full column rank and
\[
\max_{j\in \cs}\big(A^T(\bb-A\bx^k)\big)_j > \max_{l\in [N]\backslash\cs}\big(A^T(\bb-A\bx^k)\big)_l,
\]
for $k=0,1,\ldots, d-1$.
\end{corollary}

\begin{remark}\label{rem:1}
If one chooses the starting point $\bx^0={\bf 0}$ in Algorithm {\rm \ref{ap}}, then, in Theorems {\rm \ref{thm:sc}}--{\rm\ref{thm:sc-ls}} and Corollaries {\rm\ref{cor:ex}}--{\rm\ref{cor:inex}}, the condition
\[
\max_{j\in \cs}\big(A^T(\bb-A\bx^0)\big)_j > \max_{l\in [N]\backslash\cs}\big(A^T(\bb-A\bx^0)\big)_l.
\]
is replaced by
\[
\max_{j\in \cs}(A^T\bb)_j>\max_{l\in [N]\backslash\cs}(A^T\bb)_l.
\]
From the latter numerical examples, we can see that different sparse solutions to problem \eqref{pbn:ls-s} can be obtained via Algorithm {\rm \ref{ap}} with different choices of sparse $\bx^0\in\cm$ or $\bx^0={\bf 0}$.
\end{remark}

\section{Numerical experiments} \label{sec4}
In this section, we present the numerical performance of Algorithm \ref{ap} for solving problem \eqref{pbn:ls-s}. To illustrate the efficiency of our method, we compare the proposed algorithm with the maximum entropy rate approach (MEM)  in \cite{CC11}  and the projection-based gradient descent method (PG) in \cite{WW15}. All numerical tests were carried out using {\tt MATLAB R2020a} on a personal laptop with an Intel(R) Core(TM) i7--5500U CPU at 2.4 GHz  and 8GB of RAM.

In our numerical experiments, the initial point $\bx^0$ is chosen to be (a) $\bx^0={\bf 0}$ and (b) $\bx^0\in\cm$ is a random sparse $N$-vector with $s$ uniformly distributed nonzero entries, where $s=1,2$. The stopping criterion for Algorithm \ref{ap} is given by
\[
\|A_{\supp(\bx^k)}^T\br^k-({\bx}^k)^TA^T\br^k\|_2+
\|\max (A_{[N]\backslash \supp(\bx^k)}^T\br^k-({\bx}^k)^TA^T\br^k,{\bf 0})\|_2\le 10^{-7}
\]
with $\br^k=\bb-A\bx^k$ and the largest number of iterations for Algorithm \ref{ap} is set to be $m$. 

We consider the following numerical examples.

\begin{example}\label{ex2}
Consider another example in \cite{CJ12} with two genes {\rm (}$n=2${\rm )}, where the observed  transition probability matrix is given by
$$P_1=
\left[
\begin{array}{cccc}
0.1 & 0.3 & 0.2 & 0.1 \\
0.2 & 0.3 & 0.2 & 0.0 \\
0.0 & 0.0 & 0.6 & 0.4 \\
0.7 & 0.4 & 0.0 & 0.5 \\
\end{array}
\right].
$$
In this PBN,  there are $N=81$ BNs.
\end{example}


\begin{example}\label{ex4}
We consider a network in \cite{WW15} where the prescribed transition probability matrix of the PBN is given by
$$P_2=
\left[
\begin{array}{cc}
P_1 & 0  \\
0 & P_1
\end{array}
\right].$$
In this PBN, there are $6561$ BNs.
\end{example}

The numerical results for Examples \ref{ex2}--\ref{ex4} are displayed in Figures \ref{fig2-1}--\ref{fig4-1} and Tables \ref{tab2}--\ref{tab4}. Here, $\bx^\#$ denotes the computed  solution to problem \eqref{pbn:ls-s} obtained via MEM, PG, and Algorithm \ref{ap} accordingly, the symbols `{\tt Obj.}' and  `{\tt CT.}' mean the total computing time in seconds and the objective function value $\frac{1}{2}\|A\bx^\#-\bb\|_2^2$
at the final iterate of the corresponding algorithm, respectively, `${\tt sum(j)}$' is the sum  of the $j$  largest components of the computed solution $\bx^\#$ for different $j$ and `{\tt Obj(j)}' is the corresponding reconstructed objective function value $\frac{1}{2}\|\bb-A_{\pi(1:j)}\bx^\#(\pi(1:j))\|_2^2$, where $\pi=\{\pi(1),\ldots,\pi(N)\}$ is a permutation such that $$x^\#_{\pi(1)}\ge x^\#_{\pi(2)}\ge\cdots\ge x^\#_{\pi(N)}\quad \mbox{and}\quad \bx^\#(\pi(1:j))=(x^\#_{\pi(1)},\ldots, x^\#_{\pi(j)})^T\in\R^{j}.$$

We observe from Figures \ref{fig2-1}--\ref{fig4-1} that the solution obtained by Algorithm \ref{ap} are much sparser than MEM and PG. We also see from Tables \ref{tab2}--\ref{tab4} that the identified major BNs by Algorithm \ref{ap} with various initial guesses leads to much less residual than MEM and PG.
\begin{figure}[!h]
 \caption{The probability distribution $\bx^\#$ for Example \ref{ex2}.}
 \label{fig2-1}
 \centering
\includegraphics[width=0.3\textwidth]{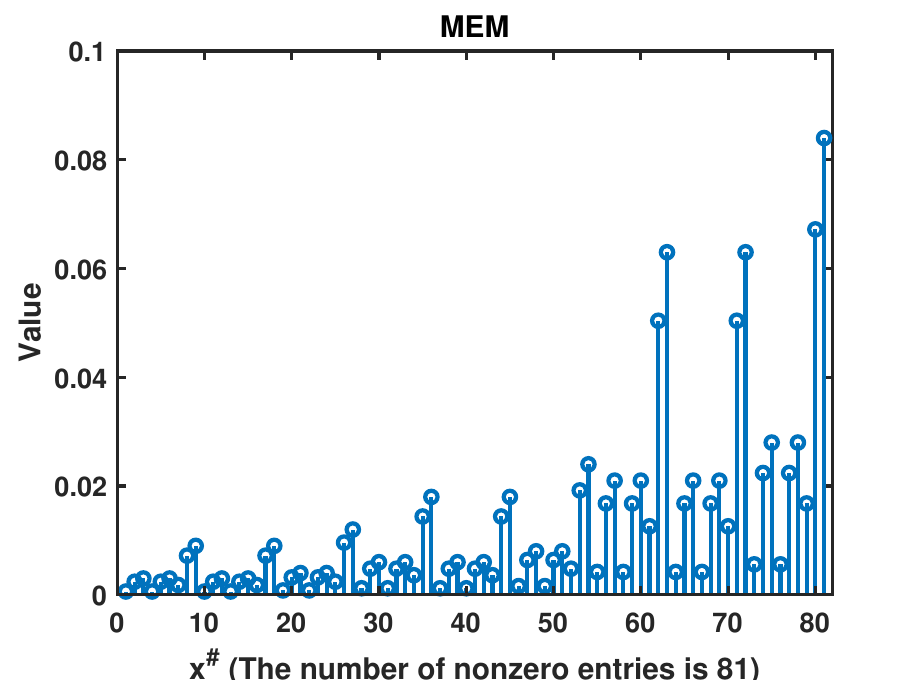}
\includegraphics[width=0.3\textwidth]{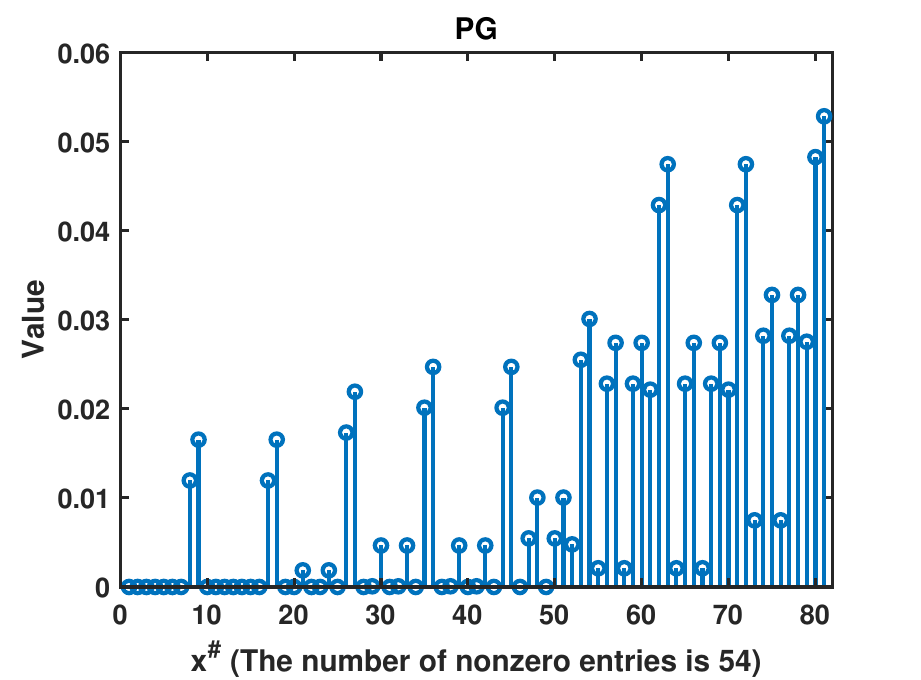}\\
\includegraphics[width=0.3\textwidth]{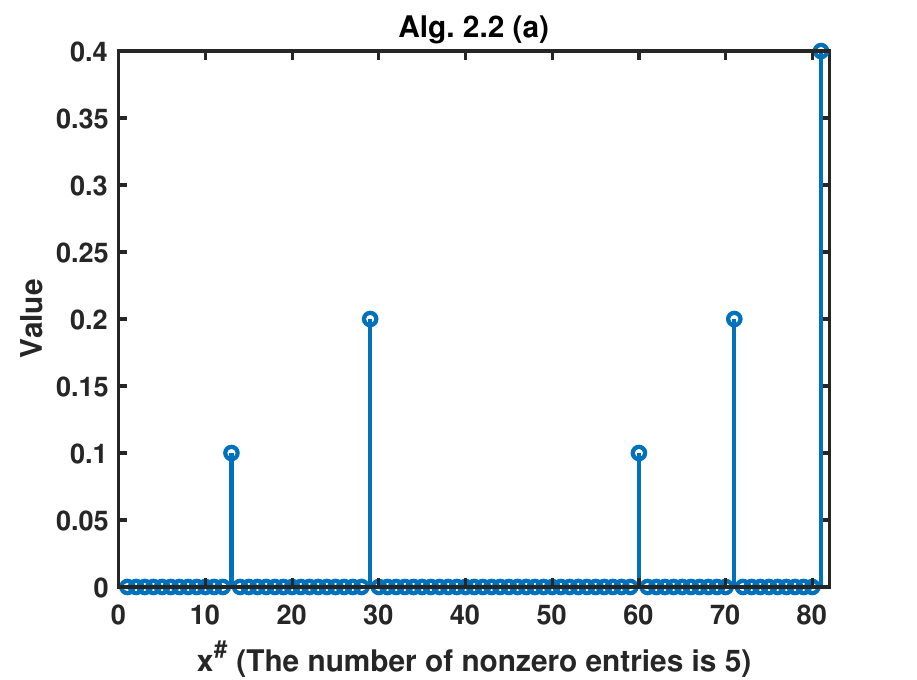}
\includegraphics[width=0.3\textwidth]{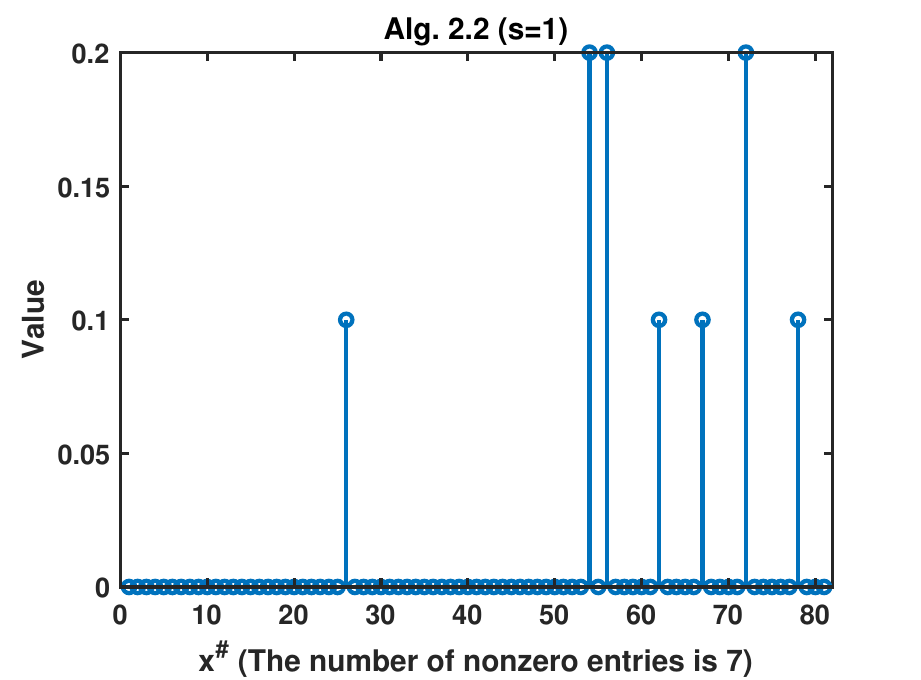}
\includegraphics[width=0.3\textwidth]{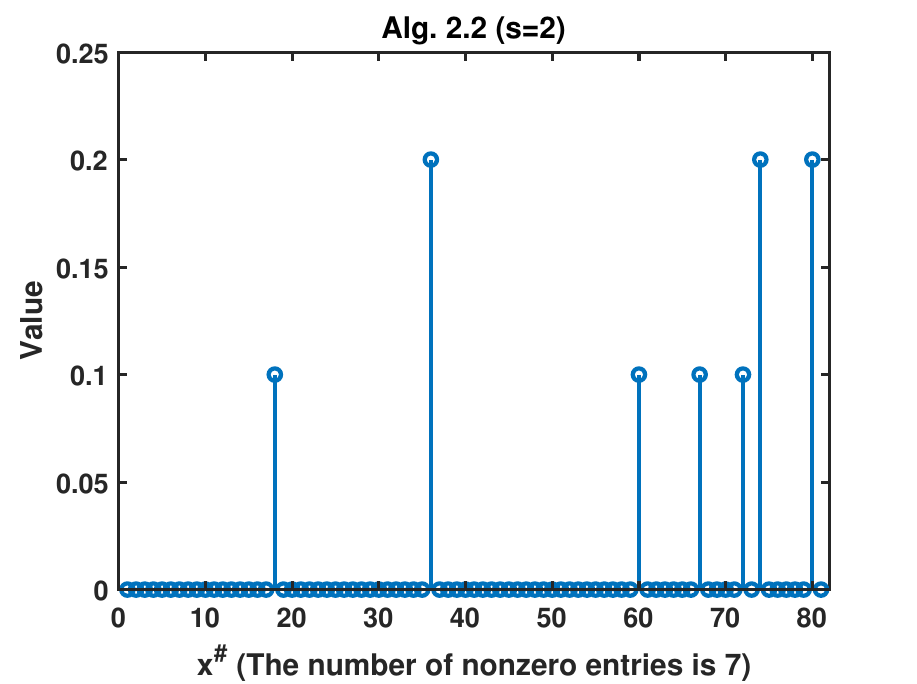}
\end{figure}

\begin{table}[!h]\renewcommand{\arraystretch}{1.0} \addtolength{\tabcolsep}{1.2pt}
\caption{Numerical results for Example \ref{ex2}.} \label{tab2}
\begin{center}
\resizebox{\textwidth}{25mm}{
 \begin{tabular}[c]{|c|l|c|l|c|l|c|l|c|l|c|}     \hline
& \multicolumn{2}{|c|}{MEM} & \multicolumn{2}{|c|}{PG} & \multicolumn{2}{|c|}{Alg. \ref{ap} (a)} & \multicolumn{2}{|c|}{Alg. \ref{ap} (s=1)} & \multicolumn{2}{|c|}{Alg. \ref{ap} (s=2)} \\ \hline
{\tt Obj.} &  \multicolumn{2}{|c|}{$1.9035\times 10^{-12}$} &  \multicolumn{2}{|c|}{$2.7759\times 10^{-8}$} &   \multicolumn{2}{|c|}{$1.9324\times 10^{-12}$} &  \multicolumn{2}{|c|}{$1.1331\times 10^{-12}$} &  \multicolumn{2}{|c|}{$1.2651\times 10^{-13}$}  \\ \hline
{\tt CT.} &  \multicolumn{2}{|c|}{$0.0039$} &  \multicolumn{2}{|c|}{$0.0022$} &  \multicolumn{2}{|c|}{$0.0274$}  &  \multicolumn{2}{|c|}{$0.0198$}  &  \multicolumn{2}{|c|}{$0.0273$}  \\ \hline
$\|\bx^\#\|_0$ &  \multicolumn{2}{|c|}{$81$} &  \multicolumn{2}{|c|}{$54$} &  \multicolumn{2}{|c|}{$5$} &  \multicolumn{2}{|c|}{$7$} &  \multicolumn{2}{|c|}{$7$}   \\ \hline\hline
 $j$ & ${\tt Obj(j)}$& ${\tt sum(j)}$  & ${\tt Obj(j)}$& ${\tt sum(j)}$  & ${\tt Obj(j)}$& ${\tt sum(j)}$  & ${\tt Obj(j)}$& ${\tt sum(j)}$   & ${\tt Obj(j)}$& ${\tt sum(j)}$  \\ \hline
$1$ & $1.1826\times 10^{0}$                       &$0.0840$  &      $1.2323\times 10^{0} $&$0.0528$  &$7.8740\times 10^{-1}$  &$0.4000$  &$1.1045\times 10^{0}$  &$0.2000$  &$1.0296\times 10^{0}$  &$0.2000$ \\ \hline
$2 $& $1.0809\times 10^{0}$                      &$0.1512$  &  $1.1578\times 10^{0} $    &$0.1011$  &$5.4772\times 10^{-1}$  &$0.6000$   &$8.3666\times 10^{-1}$  &$0.4000$ &$8.8318\times 10^{-1}$  &$0.4000$\\ \hline
$3 $&  $9.8401\times 10^{-1}$                     &$0.2142$  & $1.0837\times 10^{0}$     &$0.1486$  &$3.1623\times 10^{-1}$ &$0.8000$   &$5.4772\times 10^{-1}$ &$0.6000$  &$6.1644\times 10^{-1}$ &$0.6000$\\ \hline
$4 $& $8.8997\times 10^{-1}$                      &$0.2772$  &$1.0108\times 10^{0}$      &$0.1960$ &$2.0000\times 10^{-1}$ &$0.9000$     &$4.2426\times 10^{-1}$ &$0.7000$   &$4.8990\times 10^{-1}$ &$0.7000$    \\ \hline
$5$&   $8.1831\times 10^{-1}$                     &$0.3276$  & $9.4840\times 10^{-1}   $  &$0.2389$  &$1.9324\times 10^{-12}$ &$1.0000$   &$2.8284\times 10^{-1}$ &$0.8000$  &$3.4641\times 10^{-1}$ &$0.8000$       \\  \hline
$6$& $7.4996\times 10^{-1}$                      &$0.3780$   & $8.8778\times 10^{-1}$     &$0.2818$  &                                          &                 &$2.0000\times 10^{-1}$ &$0.9000$  &$2.0000\times 10^{-13}$ &$0.9000$       \\ \hline
$7$& $7.1145\times 10^{-1}$                       &$0.4060$  &  $8.4215\times 10^{-1} $   &$0.3146$  &                                           &                &$1.1331\times 10^{-12}$ &$1.0000$    &$1.2651\times10^{-13} $& $1.0000$    \\ \hline
$8$&   $6.7424\times 10^{-1}$                    &$0.4340$   & $7.9794\times 10^{-1}$     &$0.3474$  &                                           &               & & & & \\ \hline
$9$ & $6.4511\times 10^{-1}$                      &$0.4580$   &$7.6017\times 10^{-1}$      &$0.3775$  &                                             &               & & & & \\ \hline
$10$&   $6.1692\times 10^{-1}$                   &$0.4804$   & $7.2474\times 10^{-1}$      &$0.4057$ &                                         &                 & & & & \\ \hline
$11$& $5.8994\times 10^{-1}$                      &$0.5028$   &$6.9094\times 10^{-1}$      &$0.4339$  &                                          &                 & & & & \\ \hline
$12$&$5.6353\times 10^{-1}$                        &$0.5238$   & $6.5929\times 10^{-1}$    &$0.4615$  &                                         &                 & & & & \\ \hline
$13$&$5.3828\times 10^{-1}$                        &$0.5448$   &$6.2431\times 10^{-1}$      &$0.4889$ &                                         &                 & & & & \\ \hline
$14$& $5.1350\times 10^{-1}$                       &$0.5658$    &$5.9108\times 10^{-1}$    &$0.5163$   &                                        &                & & & & \\ \hline
$15$&  $4.9018\times 10^{-1}$                     &$0.5868$    &$5.5855\times 10^{-1}$    & $0.5437$ &                                          &                     & & & & \\ \hline
$16$& $4.6415\times 10^{-1}$                       &$0.6060$    &$5.2830\times 10^{-1}$   &$0.5711$  &                                           &                 & & & & \\ \hline
\end{tabular} }
\end{center}
\end{table}
\begin{figure}[!h]
\caption{The probability distribution $\bx^{\#}$ for Example \ref{ex4}.}
 \label{fig4-1}
 \centering
\includegraphics[width=0.3\textwidth]{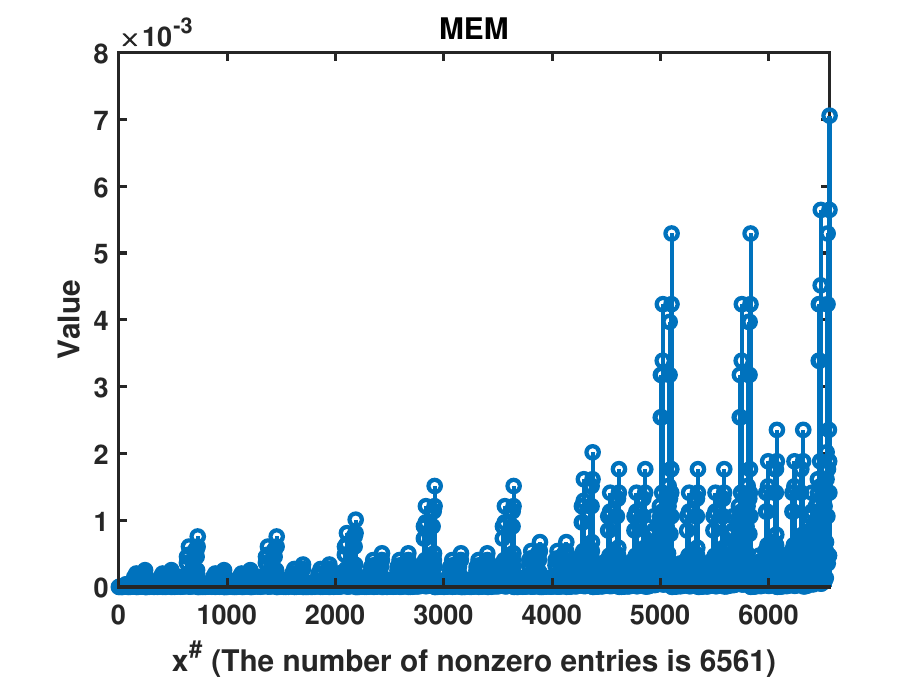}
\includegraphics[width=0.3\textwidth]{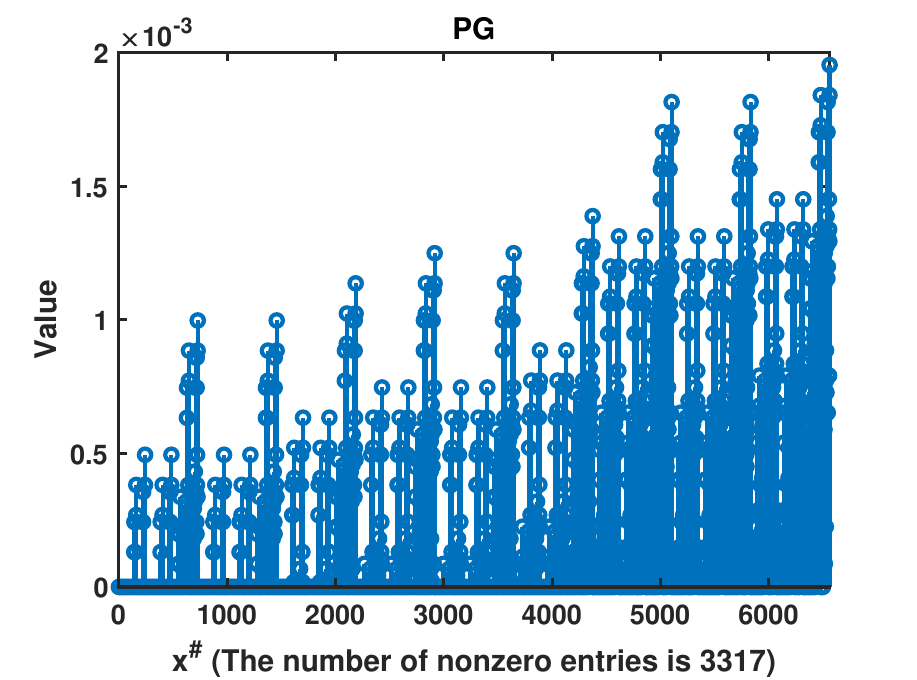}\\
\includegraphics[width=0.3\textwidth]{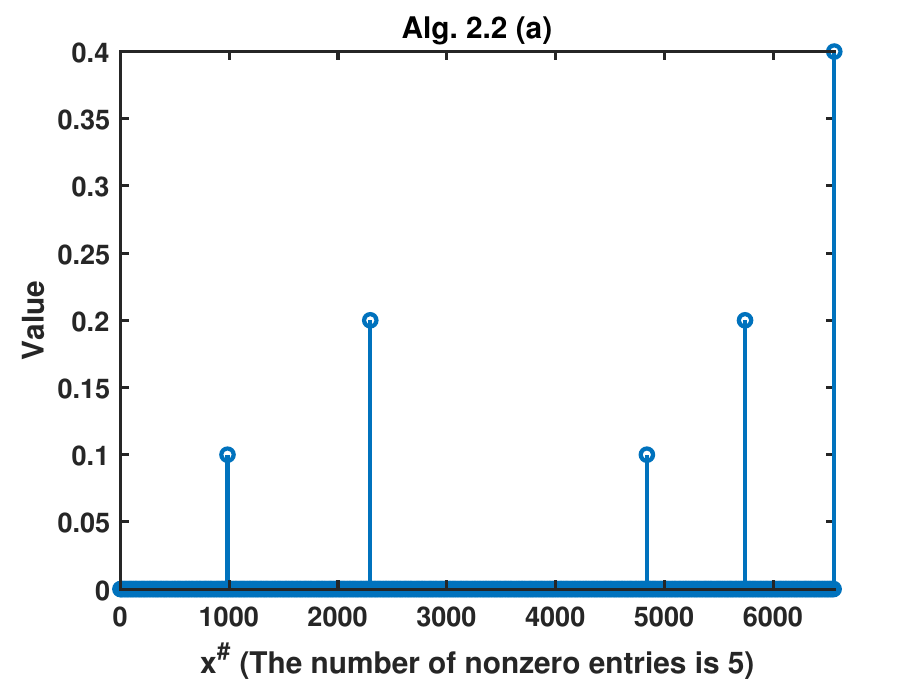}
\includegraphics[width=0.3\textwidth]{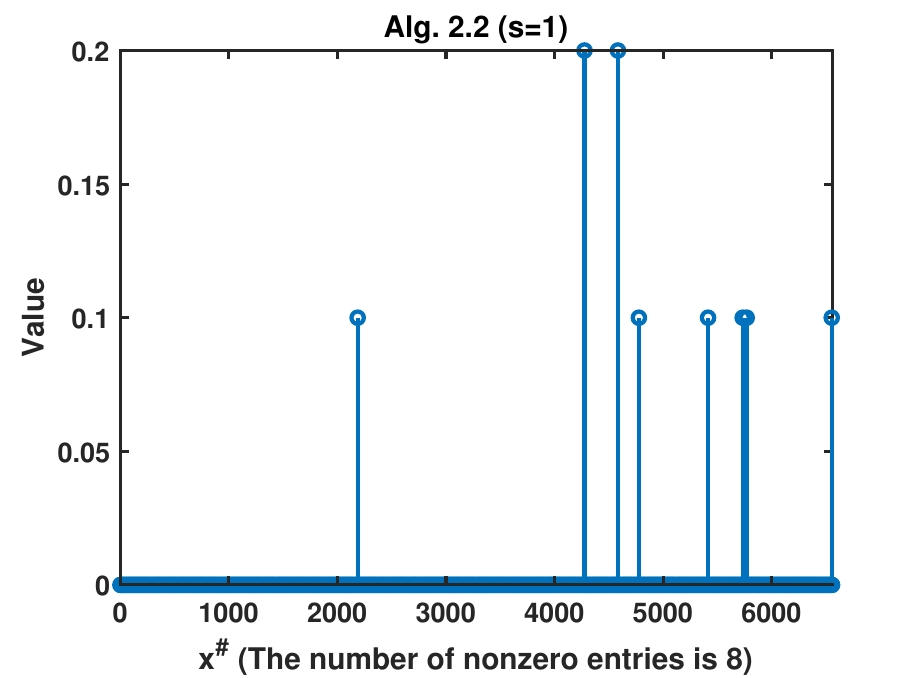}
\includegraphics[width=0.3\textwidth]{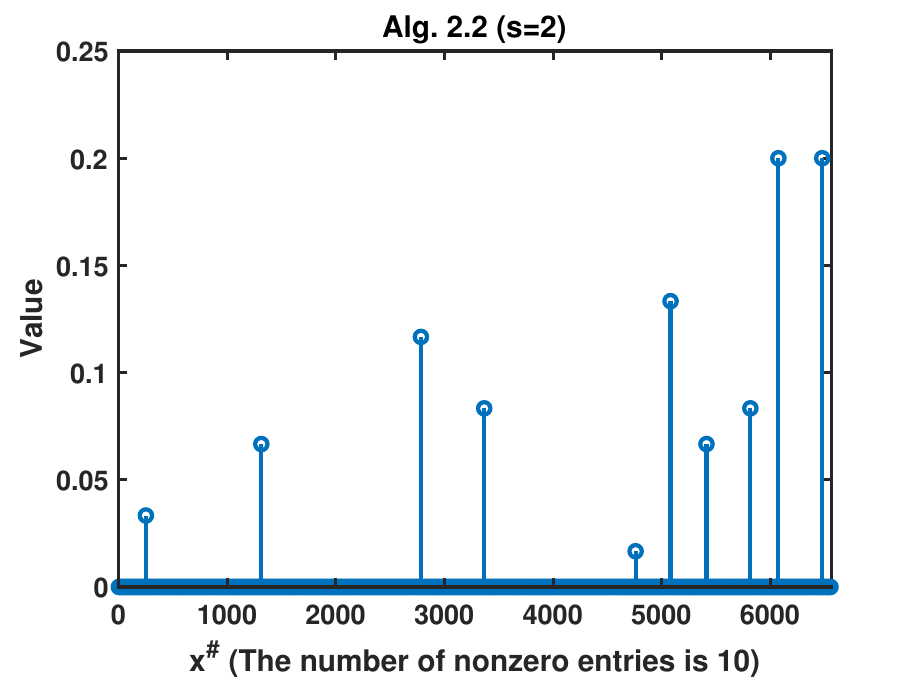}
 \end{figure}

\begin{table}[!h]\renewcommand{\arraystretch}{1.0} \addtolength{\tabcolsep}{1.0pt}
\caption{Numerical results for Example \ref{ex4}.}  \label{tab4}
\begin{center}
\resizebox{\textwidth}{25mm}{
  \begin{tabular}[c]{|c|l|c|l|c|l|c|l|c|l|c|}     \hline
& \multicolumn{2}{|c|}{MEM} & \multicolumn{2}{|c|}{PG} & \multicolumn{2}{|c|}{Alg. \ref{ap} (a)} & \multicolumn{2}{|c|}{Alg. \ref{ap} (s=1)} & \multicolumn{2}{|c|}{Alg. \ref{ap} (s=2)} \\ \hline
{\tt Obj.} &  \multicolumn{2}{|c|}{$3.6129\times 10^{-9}$} &  \multicolumn{2}{|c|}{$1.5072\times 10^{-7}$} &   \multicolumn{2}{|c|}{$1.6574\times 10^{-13}$} &  \multicolumn{2}{|c|}{$1.0598\times 10^{-13}$} &  \multicolumn{2}{|c|}{$2.3896\times 10^{-12}$}  \\ \hline
{\tt CT.} &  \multicolumn{2}{|c|}{$2.3761$} &  \multicolumn{2}{|c|}{$1.6465$} &  \multicolumn{2}{|c|}{$0.0286$}  &  \multicolumn{2}{|c|}{$0.0351$}  &  \multicolumn{2}{|c|}{$0.0407$}  \\ \hline
$\|\bx^\#\|_0$ &  \multicolumn{2}{|c|}{$6561$} &  \multicolumn{2}{|c|}{$3317$} &  \multicolumn{2}{|c|}{$5$} &  \multicolumn{2}{|c|}{$8$} &  \multicolumn{2}{|c|}{$10$}   \\ \hline\hline
$j$ & ${\tt Obj(j)}$& ${\tt sum(j)}$  & ${\tt Obj(j)}$& ${\tt sum(j)}$  & ${\tt Obj(j)}$& ${\tt sum(j)}$  & ${\tt Obj(j)}$& ${\tt sum(j)}$   & ${\tt Obj(j)}$& ${\tt sum(j)}$  \\ \hline
$1$ & $1.8489\times 10^{0}$                       &$0.0071$  &      $1.8609\times 10^{0} $&$0.0020$  &$1.1136\times 10^{0}$  &$0.4000$ &$1.6248\times 10^{0}$  &$0.2000$   &$1.4832\times 10^{0}$  &$0.2000$       \\ \hline
$2 $& $1.8359\times 10^{0}$                      &$0.0127$  &  $1.8566\times 10^{0} $    &$0.0038$  &$7.7460\times 10^{-1}$  &$0.6000$ &$1.2166\times 10^{0}$  &$0.4000$  &$1.1136\times 10^{0}$  &$0.4000$  \\ \hline
$3 $&  $1.8230\times 10^{0}$                     &$0.0183$  & $1.8524\times 10^{0}$     &$0.0056$  &$4.4721\times 10^{-1}$ &$0.8000$   &$1.0198\times 10^{0}$ &$0.5000$   &$8.6152\times 10^{-1}$ &$0.5333$\\ \hline
$4 $& $1.8108\times 10^{0}$                      &$0.0236$  &$1.8482\times 10^{0}$      &$0.0075$ &$2.8284\times 10^{-1}$ &$0.9000$&$7.8740\times 10^{-1}$ &$0.6000$        &$6.9121\times 10^{-1}$ &$0.6500$\\ \hline
$5$&   $1.7988\times 10^{0}$                     &$0.0289$  & $1.8440\times 10^{0}   $  &$0.0093$  &$1.6574\times 10^{-13}$ &$1.0000$    &$6.1644\times 10^{-1}$ &$0.7000$  &$5.3229\times 10^{-1}$ &$0.7333$      \\  \hline
$6$& $1.7867\times 10^{0}$                      &$0.0342$   & $1.8399\times 10^{0}$     &$0.0111$  &                                          &                   &$4.6904\times 10^{-1}$ &$0.8000$   &$3.8006\times 10^{-1}$ &$0.8167$\\ \hline
$7$& $1.7746\times 10^{0}$                      &$0.0395$   &  $1.8357\times 10^{0} $   &$0.0129$  &                                         &                    &$2.8284\times 10^{-1}$ &$0.9000$  &$2.6667\times 10^{-1}$ &$0.8833$   \\ \hline
$8$&   $1.7645\times 10^{0}$                    &$0.0440$   & $1.8318\times 10^{0}$     &$0.0146$  &                                         &                     &$1.0598\times 10^{-13}$ &$1.0000$  &$1.2472\times 10^{-1}$ &$0.9500$\\ \hline
$9$ & $1.7551\times 10^{0}$                      &$0.0483$   &$1.8280\times 10^{0}$    &$0.0163$  &                                           &                     &                                          &               & $4.7140\times 10^{-2}$ &$0.9833$ \\ \hline
$10$&   $1.7456\times 10^{0}$                    &$0.0525$   & $1.8242\times 10^{0}$  &$0.0180$ &                                             &                   &                                             &             & $2.3896\times 10^{-12}$  &$1.0000$ \\ \hline
$11$& $1.7362\times 10^{0}$                     &$0.0567$   &$1.8204\times 10^{0}$      &$0.0197$  &  &  & & & & \\ \hline
$12$&$1.7268\times 10^{0}$                      &$0.0610$   & $1.8166\times 10^{0}$    &$0.0214$  &  &  & & & & \\ \hline
$13$&$1.7174\times 10^{0}$                       &$0.0652$   &$1.8127\times 10^{0}$     &$0.0231$ &  &  & & & & \\ \hline
$14$& $1.7079\times 10^{0}$                     &$0.0694$    &$1.8089\times 10^{0}$   &$0.0248$   &  &  & & & & \\ \hline
$15$&  $1.6985\times 10^{0} $                      &$0.0737$    &$1.8051\times 10^{0}$   & $0.0265$ &  &  & & & & \\ \hline
$16$& $1.6891\times 10^{0}$                        &$0.0779$    &$1.8013\times 10^{0}$  &$0.0282$  &  &  & & & & \\ \hline
\end{tabular} }
\end{center}
\end{table}

To further illustrate the effectiveness of our method, in the following numerical example, we  only compare the performance of our method with that of PG for reconstructing a sparse solution to  problem \eqref{pbn:ls-s} in the least square sense since the problem size is very large and the MEM is not so effective as expected.
\begin{example}\label{ex5}
We consider a network in \cite{GC13} for modelling credit defaults, where the prescribed transition probability matrix of the PBN is given by
\[
P_3=
\left[
\begin{array}{cccccccc}
0.57&0.00&0.10&0.00&0.00&0.04&0.00&0.00\\
0.14&0.31&0.00&0.50&0.13&0.13&0.33&0.06\\
0.00&0.08&0.40&0.25&0.25&0.00&0.67&0.00\\
0.00&0.15&0.00&0.00&0.00&0.08&0.00&0.00\\
0.00&0.15&0.30&0.00&0.00&0.13&0.00&0.00\\
0.29&0.31&0.20&0.00&0.25&0.29&0.00&0.39\\
0.00&0.00&0.00&0.00&0.38&0.00&0.00&0.00\\
0.00&0.00&0.00&0.25&0.00&0.33&0.00&0.56
\end{array}
\right].
\]
In this PBN, there are $25920$ BNs.
\end{example}

The numerical results for Example \ref{ex5} are displayed in Figures \ref{fig5-1} and Table \ref{tab5}.
Figure \ref{fig5-1} shows that the least square solution generated by our method is  much sparse than PG. We also see from Table \ref{tab5} that the major BNs obtained by Algorithm \ref{ap} yields much less residual than PG.

\begin{figure}[!h]
 \caption{The probability distribution $\bx^{\#}$ for Example \ref{ex5}.}
 \centering
\includegraphics[width=0.24\textwidth]{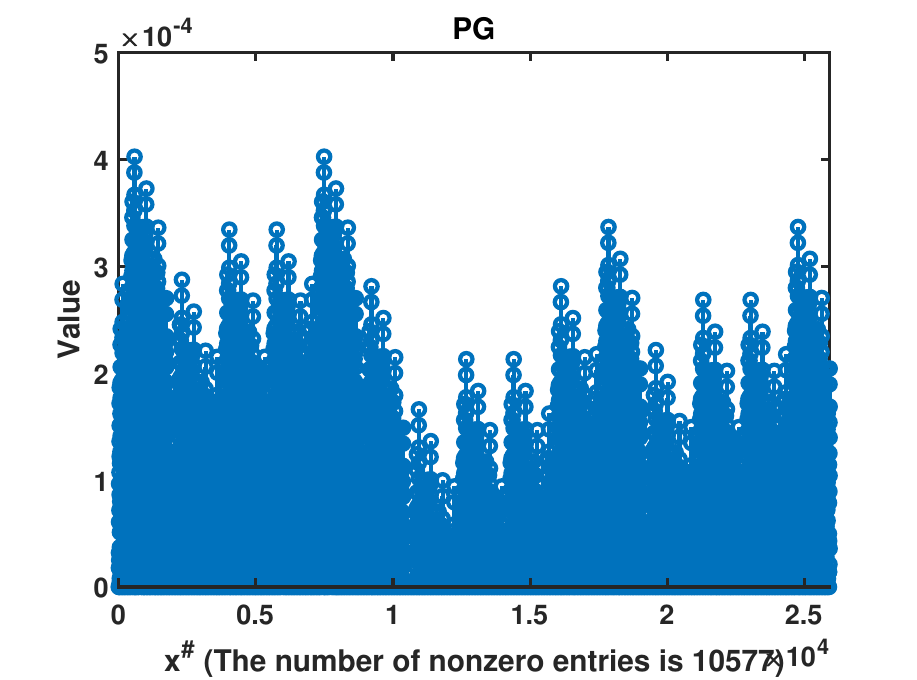}
\includegraphics[width=0.24\textwidth]{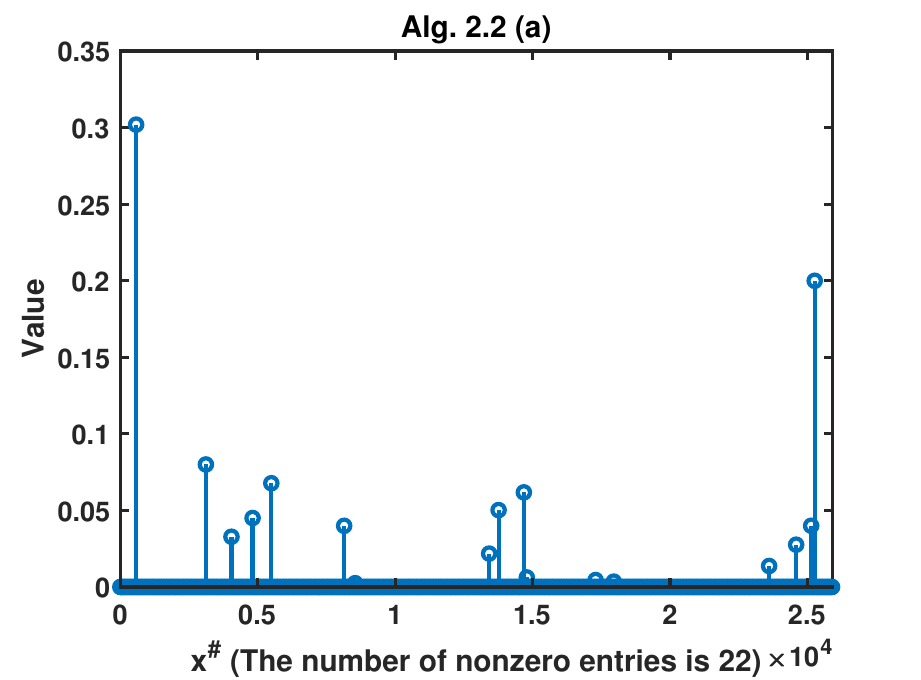}
\includegraphics[width=0.24\textwidth]{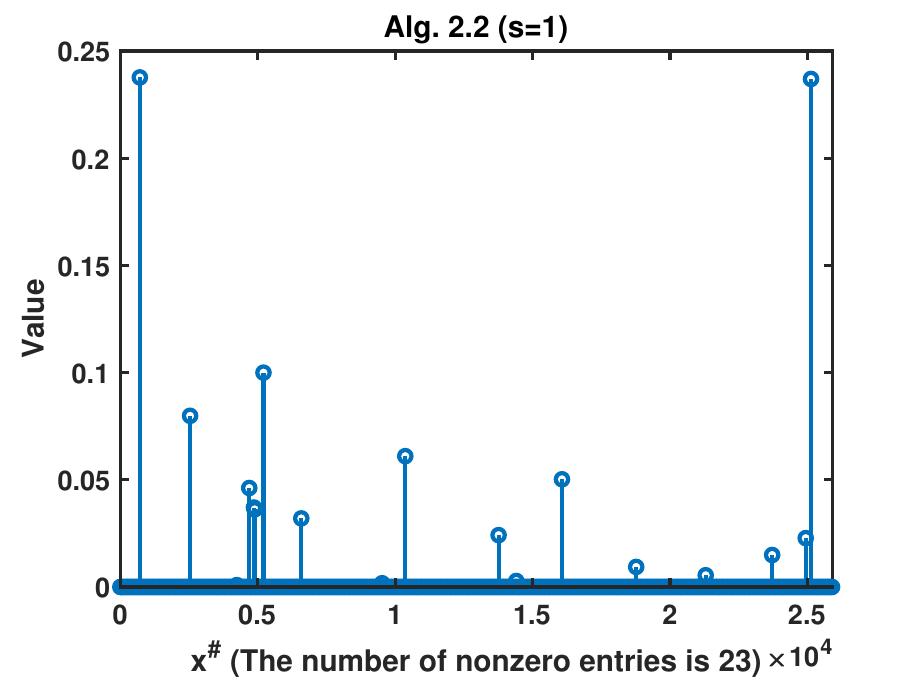}
\includegraphics[width=0.24\textwidth]{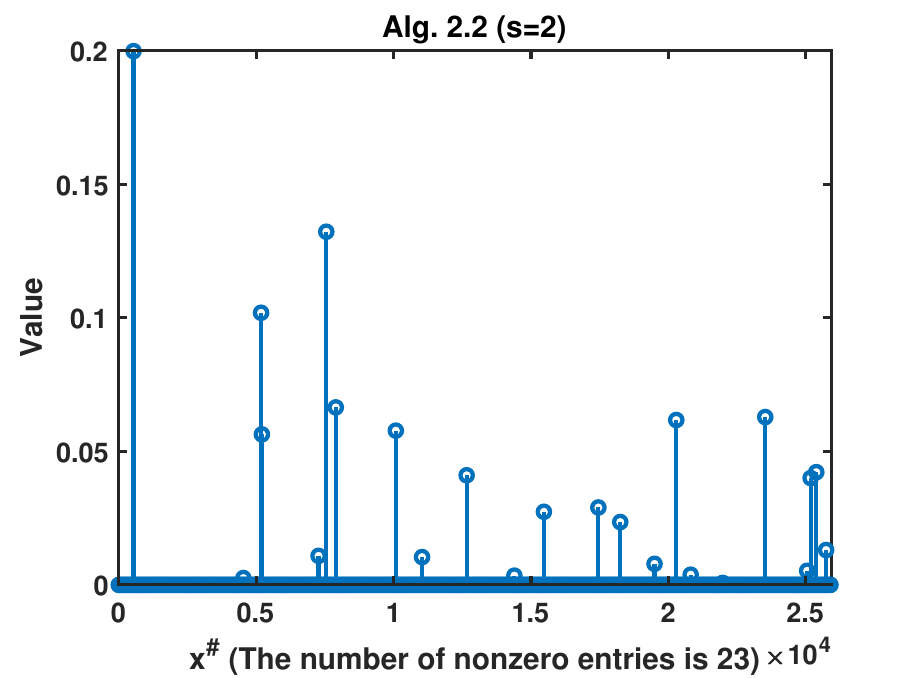}

 \label{fig5-1}
 \end{figure}

\begin{table}[!h]\renewcommand{\arraystretch}{1.0} \addtolength{\tabcolsep}{1.0pt}
\caption{Numerical results for Example \ref{ex5}.} \label{tab5}
\begin{center}
\resizebox{\textwidth}{40mm}{
  \begin{tabular}[c]{|c|l|c|l|c|l|c|l|c|}     \hline
& \multicolumn{2}{|c|}{PG} & \multicolumn{2}{|c|}{Alg. \ref{ap} (a)} & \multicolumn{2}{|c|}{Alg. \ref{ap} (s=1)} & \multicolumn{2}{|c|}{Alg. \ref{ap} (s=2)} \\ \hline
{\tt Obj.} &  \multicolumn{2}{|c|}{$7.6376\times 10^{-3}$} &  \multicolumn{2}{|c|}{$7.6376\times 10^{-3}$} &   \multicolumn{2}{|c|}{$7.6376\times 10^{-3}$} &  \multicolumn{2}{|c|}{$7.6376\times 10^{-3}$} \\ \hline
{\tt CT.} &  \multicolumn{2}{|c|}{$2.9839$} &  \multicolumn{2}{|c|}{$0.1392$} &  \multicolumn{2}{|c|}{$0.2088$}  &  \multicolumn{2}{|c|}{$0.2310$}   \\ \hline
$\|\bx^\#\|_0$ &  \multicolumn{2}{|c|}{$10577$} &  \multicolumn{2}{|c|}{$22$} &  \multicolumn{2}{|c|}{$23$} &  \multicolumn{2}{|c|}{$23$}   \\ \hline\hline
$j$ & ${\tt Obj(j)}$& ${\tt sum(j)}$  & ${\tt Obj(j)}$& ${\tt sum(j)}$  & ${\tt Obj(j)}$& ${\tt sum(j)}$  & ${\tt Obj(j)}$& ${\tt sum(j)}$     \\ \hline
 $1$   &      $1.7007\times 10^{0} $&$0.0004$  &$1.1740\times 10^{0}$  &$0.3019$ &$1.3997\times 10^{0}$  &$0.2377$   &$1.3595\times 10^{0}$ &$0.1998$      \\ \hline
$2 $   &  $1.6998\times 10^{0} $    &$0.0008$  &$8.5697\times 10^{-1}$  &$0.5019$ &$9.4187\times 10^{-1}$  &$0.4747$  &$1.1870\times 10^{0}$ &$0.3319$  \\ \hline
$3 $ & $1.6990\times 10^{0}$     &$0.0012$  &$7.2396\times 10^{-1}$ &$0.5819$   &$7.7130\times 10^{-1}$ &$0.5746$ &$1.0009\times 10^{0}$ &$0.4338$ \\ \hline
$4 $ &$1.6982\times 10^{0}$      &$0.0016$ &$6.1625\times 10^{-1}$ &$0.6497$   &$6.2345\times 10^{-1}$ &$0.6545$  &$8.6691\times 10^{-1}$ &$0.5002$      \\ \hline
$5$  & $1.6974\times 10^{0}   $  &$0.0019$  &$5.2749\times 10^{-1}$ &$0.7114$  &$5.1882\times 10^{-1}$ &$0.7155$  &$7.5744\times 10^{-1}$ &$0.5631$      \\  \hline
$6$  & $1.6966\times 10^{0}$     &$0.0023$  &$4.3873\times 10^{-1}$ &$0.7616$ &$4.3079\times 10^{-1}$ &$0.7657$ &$6.5260\times 10^{-1}$ &$0.6248$\\ \hline
$7$  &  $1.6958\times 10^{0} $   &$0.0027$  &$3.5896\times 10^{-1}$ &$0.8067$   &$3.4471\times 10^{-1}$ &$0.8119$ &$5.4269\times 10^{-1}$ &$0.6825$ \\ \hline
$8$  & $1.6950\times 10^{0}$     &$0.0031$  &$2.7739\times 10^{-1}$ &$0.8466$   &$2.7640\times 10^{-1}$ &$0.8490$  &$4.4473\times 10^{-1}$ &$0.7388$ \\ \hline
$9$  &$1.6943\times 10^{0}$      &$0.0034$  &$2.1471\times 10^{-1}$ &$0.8866$    &$2.0859\times 10^{-1}$ &$0.8852$     &$3.6801\times 10^{-1}$ &$0.7810$  \\ \hline
$10$ & $1.6935\times 10^{0}$      &$0.0038$ &$1.5347\times 10^{-1}$ & $0.9194$   &$1.5972\times 10^{-1}$ &$0.9172$   &$3.0392\times 10^{-1}$ &$0.8221$\\ \hline
$11$  &$1.6928\times 10^{0}$      &$0.0041$  &$1.0387\times 10^{-1}$ &$0.9470$     &$1.2150\times 10^{-1}$ &$0.9414$   &$2.3503\times 10^{-1}$ &$0.8621$ \\ \hline
$12$  & $1.6920\times 10^{0}$    &$0.0045$  &$6.6657\times 10^{-2}$ &$0.9688$     &$7.2729\times 10^{-2}$ &$0.9642$  &$1.9174\times 10^{-1}$ &$0.8910$\\ \hline
$13$ &$1.6912\times 10^{0}$       &$0.0049$ &$3.6778\times 10^{-2}$ &$0.9825$      &$4.4031\times 10^{-2}$ &$0.9791$  &$1.4679\times 10^{-1}$ &$0.9184$\\ \hline
$14$ &$1.6905\times 10^{0}$      &$0.0052$   &$2.6746\times 10^{-2}$  &$0.9887$    &$2.6876\times 10^{-2}$ &$0.9885$    &$1.0694\times 10^{-1}$ &$0.9419$ \\ \hline
$15$ &$1.6898\times 10^{0}$      & $0.0056$ &$1.8270\times 10^{-2}$ &$0.9933$       &$1.6651\times 10^{-2}$ &$0.9939$ &$8.5392\times 10^{-2}$ &$0.9550$ \\ \hline
$16$  &$1.6890\times 10^{0}$         &$0.0059$  &$1.2374\times 10^{-2}$  &$0.9968$      &$1.1768\times 10^{-2}$ &$0.9967$   &$7.1039\times 10^{-2}$ &$0.9658$\\ \hline
$17$ &$1.6883\times 10^{0}$      & $0.0063$ &$8.3336\times 10^{-3}$ &$0.9993$       &  $9.2747\times 10^{-3}$  & $0.9985$&$5.1373\times 10^{-2}$ &$0.9763$ \\ \hline
$18$ &$1.6876\times 10^{0}$       &$0.0066$ &$7.8849\times 10^{-3}$ &$0.9997$      &$8.1940\times 10^{-3}$ &$0.9994$  &$3.4444\times 10^{-2}$ &$0.9841$\\ \hline
$19$ &$1.6869\times 10^{0}$      &$0.0070$   &$7.7370\times 10^{-3}$  &$0.9999$    &$7.9249\times 10^{-3}$ &$0.9997$    &$2.5720\times 10^{-2}$ &$0.9894$ \\ \hline
$20$ &$1.6862\times 10^{0}$      & $0.0073$ &$7.6775\times 10^{-3}$ &$0.99995$       &$7.7959\times 10^{-3}$ &$0.9998$ &$1.8746\times 10^{-2}$ &$0.9932$ \\ \hline
$21$  &$1.6855\times 10^{0}$         &$0.0076$  &$7.6503\times 10^{-3}$  &$0.99998$      &$7.6777\times 10^{-3}$ &$0.9999$   &$1.2993\times 10^{-2}$ &$0.9967$\\ \hline
$22$ &$1.6848\times 10^{0}$      & $0.0080$ &$7.6376\times 10^{-3}$ &$1.0000$       &$7.6563\times 10^{-3}$&$0.99998$ &$8.4447\times 10^{-3}$ &$0.9993$ \\ \hline
$23$  &$1.6841\times 10^{0}$         &$0.0083$  &                                   &                       &$7.6376\times 10^{-3}$ &$1.0000$   &$7.6376\times 10^{-3}$ &$1.0000$\\ \hline
\end{tabular} }
\end{center}
\end{table}

\section{Concluding remarks} \label{sec5}
Several numerical methods have been developed for the construction of sparse probabilistic Boolean networks. However, few greedy methods were explored. In this paper, we propose a greedy-type method, a modified  orthogonal matching pursuit,  for solving the inverse problem. We derive some conditions such that, given the transition probability matrix, our method can recover a sparse  probabilistic Boolean network exactly or in the least square sense. Numerical experiments show that our method is very effective in terms of sparse recovery. An interesting question is how to analyze the exact  sparse recovery condition in terms of the  coherence as in \cite{T04}. This needs further study.

\end{document}